\documentclass[11pt]{amsart}
\usepackage{amssymb,latexsym,amsmath}

\usepackage{accents} 
\usepackage{url}

\usepackage[all,2cell,ps]{xy}

\theoremstyle{plain}
\newtheorem{theorem}{Theorem}[section]
\newtheorem{proposition}[theorem]{Proposition}
\newtheorem{lemma}[theorem]{Lemma}
\newtheorem{corollary}[theorem]{Corollary}
\newtheorem{fact}[theorem]{Fact}

\theoremstyle{definition}

\newtheorem{problem}[theorem]{Problem}

\newcommand{\syn}{\operatorname{syn}}
\newcommand{\term}{\operatorname{term}}

\newcommand{\A}{\mathbf A}
\newcommand{\B}{\mathbf B}
\newcommand{\T}{\mathbf T}

\newcommand{\V}{\mathcal V}
\renewcommand{\H}{\mathcal H}

\title{Profiniteness in finitely generated varieties is undecidable }

\author{Anvar M. Nurakunov}

\address{Institute of Mathematics, National Academy of Sciences of the Kyrgyz Republic,
pr. Chu 265a, Bishkek, 720071 Kyrgyzstan}
\email{a.nurakunov@gmail.com}

\author{Micha{\l} M. Stronkowski}

\address{Faculty of Mathematics and Information Sciences,
Warsaw University of Technology, ul. Koszykowa 75, 00-662
Warsaw, Poland}
\email{m.stronkowski@mini.pw.edu.pl}



\thanks{The first author is supported by Alexander von Humboldt Foundation (Fellowship 2017) and Ministry of Education and Science of Kazakhstan.
The second author is supported by the Polish National Science Centre grant no. DEC-2011/01/D/ST1/06136.  
}

\begin{document}

\begin{abstract}
Profinite algebras are exactly those that are isomorphic to inverse limits  of finite algebras. Such algebras are naturally equipped with Boolean topologies. A variety $\V$ is standard if every Boolean topological algebra with the algebraic reduct in $\V$ is profinite.   

We show that there is no algorithm which takes as input a finite algebra $\A$ of a finite type and decide whether the variety  $\V(\A)$ generated by $\A$ is standard.
We also show the undecidability of some related properties. In particular, we solve a problem posed by Clark, Davey, Freese and Jackson.

We accomplish this by combining two results. The first one is Moore's result saying that there is no algorithm which takes as input a finite algebra $\A$ of a finite type and decides whether $\V(\A)$ has definable principal subcongruences.  The second is our result saying that possessing definable principal subcongruences yields possessing finitely determined syntactic congruences for varieties. The latter property is known to yield standardness.
\end{abstract}

\maketitle

\section{Introduction}

\subsection{What we prove}

Let $\V$ be a variety (an equationally defined class of algebras, see Section \ref{sec:: Tolbox} for definitions). We consider the class of $\V_{Bt}$ of Boolean topological algebras with the algebraic reducts in $\V$, and the class $V_{Bc}$ of profinite algebras (consider also as Boolean topological algebras) with the algebraic reducts in $\V$. The class $\V_{Bc}$ is called the \emph{Boolean core} of $\V$. We call $\V$ \emph{standard} provided that $\V_{Bt}=V_{Bc}$, i.e., when all Boolean topological algebras with the algebraic reducts in $\V$ are profinite. Let us formulate our main result.

\begin{theorem}\label{thm::main standard}
There is no algorithm which decides if a given finite algebra of a finite type generates a standard variety.
\end{theorem}

The most widely applicable condition forcing standardness for varieties is possessing finitely determined syntactic congruences (FDSC) \cite{CDFJ04}. This prompted the authors of \cite{CDFJ04} to formulate the following question.

\begin{problem}[\protect{\cite[Problem 9.3]{CDFJ04}}]\label{problem::FDSC}
Is there an algorithm to decide if a given finite algebra of finite type generates a variety with finitely determined syntactic congruences?
\end{problem}   
We present the answer to this question.

\begin{theorem}\label{thm::main FDSC}
There is no algorithm which decides if a given finite algebra of a finite type generates the variety with finitely determined syntactic congruences.
\end{theorem}

The standardness of $\V$  may be restated as the property that profinite algebras from $\V$ (equipped with a natural topology) may be distinguished among all Boolean topological algebras of the same type as $\V$ just by checking the satisfaction of identities defining $\V$. In particular, we can do it without referring to topology. We may then ask when this distinction can be done within a particular logic and without referring to topology. The case of first-order logic was addressed in \cite{CDJP08}. With respect to this issue, we obtain the following fact.

\begin{theorem}\label{thm::main FO}
There is no algorithm which decides if a given finite algebra of a finite type generates the variety with the Boolean core definable (relative to the class of all Boolean topological algebras of the same type as $\V$) by a set of sentences in a first-order logic.
\end{theorem}

\subsection{How we prove it}

A general strategy for Problem \ref{problem::FDSC} was already proposed in \cite{CDFJ04}. It is based on McKenzie's construction \cite{McK96}. It effectively assigns to each Turing machine $\mathbf T$ the algebra $\sf A(\mathbf T) $ such that $\mathbf T$ halts iff there is a finite bound on the cardinality of subdirectly irreducible algebras in the variety ${\sf V}(\sf A(\mathbf T))$ generated by $\sf A(\mathbf T)$. In fact, if $\T$ does not halt, then, up to term equivalence, the particular subdirectly irreducible algebra $\mathbf Q_\omega$ (described in Section \ref{sec:: undecidability}) is in ${\sf V}(\sf A(\mathbf T))$. The algebra $\mathbf Q_\omega$ admits a compatible Boolean topology. Thus $\mathbf Q_\omega$ with this topology belongs to ${\sf V}({\sf A}(\T))_{Bt}$ and does not belong to ${\sf V}({\sf A}(\T))_{Bc}$. Hence it witnesses that ${\sf V}(\sf A(\mathbf T))$ is not standard. In particular, ${\sf V}(\sf A(\T))$ does not have FDSC \cite[Example 7.7]{CDFJ04}. 

What remains to be proved is that ${\sf V}(\sf A(\T))$ has FDSC if $\T$ halts. However, recently Moore verified that it is not true \cite{Moo16}. 
In the same paper he also showed that ${\sf V}(\sf A(\T))$ does not have  definable principal subcongruences (DPSC) (see Section \ref{sec::TFPSC} for the definition). This property was invented in order to give the most elegant proof  of Baker's  finite basis theorem \cite{BW02} saying that every finitely generated congruence-distributive variety of a finite type is finitely axiomatizable \cite{Bak77}. (The idea of Baker and Wang was subsequently used also in the context of quasivariety \cite{NS09} and of deductive systems \cite{NS13}.)

In  \cite{Moo15} Moore managed to modify the algebra $\sf A(\T)$ by adding one basic operation. The resulting algebra $\sf A'(\T)$ allows him to obtain the following fact: ${\sf V}({\sf A'}(\T))$ has DPSC iff $\T$ halts. Consequently, he proved that the problem whether the variety generated by a given finite algebra of finite type has DPSC is undecidable.  In \cite{Moo16} he suggested that the construction of $\sf A'(\T)$ would be used in Problem \ref{problem::FDSC}. Indeed, again, up to term equivalence, the algebra $\mathbf Q_\omega$ belongs to ${\sf V}({\sf A'}(\T))$.  The author also wrote that a detailed analysis of polynomials in ${\sf V}({\sf A'}(\T))$ might be needed in order to show that ${\sf V}({\sf A'}(\T))$ has FDSC when $\T$ halts.

It led us to consider an idea that having DPSC may yield having FDSC in general. If it is true, then actually no additional analysis of the variety ${\sf V}({\sf A'}(\T))$ is needed. Let us look at both properties from a common perspective of defining principal congruences. The principal congruence $\theta(a,b)$ generated by a pair $(a,b)$ in an algebra $\A$ is a transitive closure of the relation $R$ which is the carrier of the subalgebra $\mathbf R$ of $\A^2$ generated by $\{(a,b),(b,a)\}\cup \{(c,c)\mid c\in A\}$. This fact may be expressed in the following way: a pair $(c,d)$ is in $\theta(a,b)$  iff there is a \emph{congruence formula} (see Section \ref{sec:: Tolbox} for formal definitions) witnessing it. These are two dimensional objects. They have {\it length} which expresses how many compositions of the relation $R$ we have to use. They also have {\it depth} which expresses, roughly speaking, how complex terms we have to use when generating $\mathbf R$. A variety $\V$ has  definable principal congruences (DPC) if there is a finite bound on length and  depth in a whole $\V$. There are many ways to weaken this property. It appears that $\V$ has FDSC exactly when there is a finite bound only on depth (we assume that the type of $\V$ is finite) and length is arbitrary. And $\V$ has DPSC when we have length and depth bounded by a finite number just for \emph{some} principal congruences. 

At first glance, it seems that the properties of having FDSC and DPSC are incomparable. The group $\mathbf S_3$ of all permutation on a three element set generates the variety with FDSC (every variety of groups has FDSC) and without DPSC \cite{BW02}. Also, every finite non-distributive lattice generates the variety without DPC \cite{McK78} but, by the main result in \cite{DJMM08,Wan90} (see also our Corollary \ref{fg+cd=>FDSC}), with FDSC. But an example of a variety with DPSC and without FDSC is unknown. 

The main technical contribution of this paper is the introduction of a common weakening of the properties of  having FDSC and having DPSC. When a variety has this property, we say that it has term finite principal subcongruences (TFPSC).  We show that having TFPSC is equivalent to having FDSC for varieties (Theorem \ref{thm:: FDPC <=> TFPSC}). 
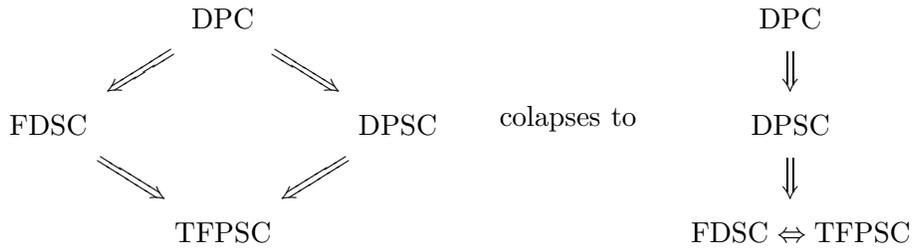
\[
\xymatrix@=15pt@M=9pt{
 & \rm DPC  \ar@{=>}[dl]  \ar@{=>}[dr] & \\
\rm FDSC  \ar@{=>}[dr] & & \rm DPSC  \ar@{=>}[dl] \\
& \rm TFPSC & 
}
\quad\text{ \raisebox{-37pt}{colapses to} }
\xymatrix@=15pt@M=9pt{
  \rm DPC  \ar@{=>}[d]   \\
\rm  \rm DPSC  \ar@{=>}[d] \\
 \;\;\,\rm FDSC \Leftrightarrow TFPSC 
}
\]

Thus it is indeed the case that having DPSC yields having FDSC. This fact, Moore's result and the analysis of $\mathbf Q_\omega$ lead to a negative answer for the question in Problem \ref{problem::FDSC}. Thus Theorems \ref{thm::main standard} and \ref{thm::main FDSC} hold. A bit of additional work on algebra $\mathbf Q_\omega$, based on a technique from \cite{CDJP08}, leads to the conclusion that the Boolean core of ${\sf V}({\sf A'}(\T))$ is not axiomatizable in first-order logic when $\T$ does not halt. This gives Theorem \ref{thm::main FO}.



\subsection{Why we prove it; A bit of history}

Let $\undertilde\A$ be a finite structure (a set with operations and relations of finite arities) equipped with a discrete topology. The class ${\sf IS_CP^+}(\undertilde\A)$ of all isomorphic images of closed subalgebras of non-zero direct powers of $\undertilde\A$ is called the \emph{topological prevariety generated by $\undertilde\A$}. Such classes appear in duality theory as dual categories \cite{CD98,Dav93}. For example, in Stone duality for Boolean algebras as $\A$ we take a two element set (no operations and relations) and then ${\sf IS_CP^+}(\undertilde\A)$ is the class of Boolean topological spaces (aka Stone spaces) \cite{Sto36}. In Pristley duality for bounded distributive lattices as $\A$ we take a two element linearly ordered set  and then ${\sf IS_CP^+}(\undertilde\A)$ is the class of Pristley spaces 
\cite{Pri70}.
In Pontryagin duality restricted to abelian groups of exponent at most $m$ we may take the group $\mathbb Z_m$ of integers modulo $m$. Then ${\sf IS_CP^+}(\undertilde{\mathbb Z_m})$ is the class of Boolean topological abelian groups of exponent at most $m$ \cite{Pon66}.

It is a convenient situation when we have a {\it good} description of members of ${\sf IS_CP^+}(\undertilde\A)$. A general scheme, based on the idea from \cite{Tay77}, for axiomatizations of topological varieties was given in \cite{CK84}. However, it involves infinite, possibly difficult to work with, expressions. They refer to topological properties of the structures under considerations. Consequently, such axiomatizations are not considered as {\it good}. The property of standardness, introduced in  \cite{CDHPT03} was intended as a formalization of the idea of good behavior for topological prevarieties with respect to the axiomatization. 

A  theory for standardness was developed in \cite{CDFJ04}, and the following general sufficient condition was presented. (Standardeness was defined there as the property for topological prevarieties, not for quasivarieties.) 

\begin{theorem}[\protect{\cite[FDSC-HSP Theorem 4.3]{CDFJ04}}]\label{thm::FDSC-HSP}
Let $\mathcal Q$ be the quasivariety (a universal Horn class with a trivial algebra) generated by a finite algebra. If  $\mathcal Q$ is a variety with FDSC, then $\mathcal Q$ is standard
\end{theorem}
Stone duality and restricted Pontryagin duality fall into the scope of this theorem. But Pristley duality, even if we extend this theorem to relational case, does not \cite{Str80}. We will use a slight extension of the FDSC-HSP theorem, see Theorem \ref{thm:: FDSC => standard}.

In \cite{CDFJ04} a main problem in this area was formulated.

\begin{problem}[\protect{\cite[Problem 9.1]{CDFJ04}}]\label{problem::standardness}
Is there an algorithm to decide if a given finite algebra of finite type generates a standard universal Horn class?
\end{problem}

Problem \ref{problem::standardness} seems to be very difficult. It is open even when we restrict it to finite lattices \cite[Problem 2]{CDJP08}. Thus the authors of \cite{CDFJ04} also formulated a simpler Problem \ref{problem::FDSC}, together with some hints how to approach it.

In \cite{CDJP08} two extensions were made. Firstly, not only finitely generated, but also topological prevarieties generated by sets of finite structures are considered. Secondly, a weakening of standardness to first-order axiomatizations, as still {\it good} description, was proposed. General techniques for disproving standardness and first-order axiomatization were presented. In particular, it appears that Pristley duality is as bad as it can be: the topological prevariety of Pristley spaces in not first-order axiomatizable \cite[Example 6.2]{CDJP08}. An analogous problem to Problem \ref {problem::standardness} for first-order aximatizability is formulated \cite[Problem 1]{CDJP08}. It is also suggested that standardness may yield finite aximatizability for finitely generated universal Horn classes. There are known standard finitely generated varieties which are not finitely axiomatizable (for instance of semigroups). However, such universal Horn classes are unknown \cite[Problem 3]{CDJP08}. Note  that the converse implication does not hold. The class of ordered sets is not standard \cite{Str80}. Moreover, a non-standard finitely axiomatizable universal Horn class generated by a four-element algebra in presented in \cite[Section 5]{Jac08}.

\vspace{1em}
We would like to point out the following quotation of Johnstone  \cite[Subsection VI.2.6]{Joh86} 
\begin{quote}
The question thus arises: given an algebraic category $\bf A$, when
can we say that every Stone topological $\bf A$-algebra is profinite? It seems
hard to give a simple condition on $\bf A$ which is both necessary and sufficient;
\end{quote}

The meaning of the adjective {\it simple} is unclear. But one could argue that simplicity should yield decidability for varieties presented in a finitary way. Thus Theorem \ref{thm::main standard} confirms Johnstone's supposition in case of presenting a variety by a finite generating algebra. Yet, there is another finitary way of presenting varieties: by a finite set of defining identities. For this case Jackson proved an analogous theorem.

\begin{theorem}[\protect{\cite[Theorem 6.1]{Jac08}}]
There is no algorithm which decides if a given finite set of identities defines a standard variety or a variety with FDSC.
\end{theorem}

\section{Toolbox}
\label{sec:: Tolbox}
Here we provide the required definitions and facts. If necessary, the reader may consult \cite{BS81} for notions in universal algebra and \cite{CDJP08,Joh86} for notions on topological prevariety (note that adopt the terminology from \cite{CDJP08}.

\subsection{Algebras, topological algebras and Boolean cores}
Let us fix an arbitrary finite algebraic type $\sigma$, i.e., a finite set of symbols of basic operations with ascribed arities. Except Section \ref{sec:: undecidability}, all (topological) algebras will be of type $\sigma$. 

A \emph{universal Horn class} of algebras is a class definable by (first-order) universal Horn sentences, i.e., universal sentences without occurrences of conjunction and with at most one positive literal. Equivalently, a class of algebras is a universal Horn class iff it is closed under the formation of isomorphic images, subalgebras and direct products over a nonempty indexing sets  \cite[Theorem V.2.23]{BS81}. Our results refers to classes closed under homomorphic images, in particular to varieties. A \emph{variety} is a class of algebras definable by \emph{identities}, i.e., sentences which are universally quantified equations. Equivalently, a class of algebras is a variety iff it is closed with respect to the formation of homomorphic images, subalgebras and direct products over arbitrary indexing sets  \cite[Theorem II.11.9]{BS81}.

A \emph{topological} algebra $\undertilde\A$ is a pair $(\A,{\mathcal T})$ where $\A$ is an algebra of type $\sigma$ and ${\mathcal T}$ is a topology such that every basic operation $\omega\colon\A^n\to\A$ is a continuous map with respect to topological spaces $(A,{\mathcal T})$ and its power $(A,{\mathcal T})^n$ (we also say that $\mathcal T$ is \emph{compatible} with basic operations of $\A$). 

In this paper all considered topological spaces are \emph{Boolean}. This means that they are Hausdorff, compact and totally disconnected. Topological algebras with such topologies are called \emph{Boolean topological algebras}. Most commonly encountered Boolean topological algebras are profinite (groups, rings, semigroups, distributive lattices, Heyting algebras, closure algebras, all of them have FDSC \cite{CDFJ04}). It means that they are isomorphic to inverse limits of finite algebras. However, not all Boolean topological algebras are profinite. For instance, one may take an infinite subdirectly irreducible algebra (i.e., having a least congruence which is not the identity relation) admitting a Boolean topology. An example of such unary algebra is given in \cite[Point VI.2.5]{Joh86} and of modular lattice  in \cite{Cli79}, see also \cite[Example 2.10]{CDJP08}. One such example, presented in Section \ref{sec:: undecidability}, will be crucial for us.

Following \cite{CDJP08}, we call a class of Boolean topological algebras which is closed under the formation of isomorphic images, topologically closed subalgebras and direct products over a nonempty indexing set a \emph{topological prevariety}.  (The fact that we disallow empty indexing set is not relevant here. Indeed, the difference is only with adding or excluding a trivial algebra. It has its origin in duality theory, where this definition is natural.)

For every class $\mathcal K$ of Boolean topological algebras there exists a smallest topological previariety containing $\mathcal K$. It consists of topological algebras which are isomorphic to closed subalgebras of non-zero products of members of $\mathcal K$. 

With a given universal Horn class  $\H$ (in particular, with a variety) we associate two topological prevarieties: The first one, denoted by $\H_{Bt}$ consists of all Boolean topological algebras with the algebraic reducts in $\H$. The second one, denoted by $\H_{Bc}$ and called the \emph{Boolean core of} $\H$, is the topological prevariety generated by finite members of $\H$, each of them considered as a Boolean topological algebra with the discrete topology. In general, we have the inclusion $\H_{Bc}\subseteq\H_{Bt}$. We say that $\H$ is \emph{standard} if $\H_{Bt}=\H_{Bc}$. 

It appears that the members of $\H_{Bc}$ are exactly those Boolean topological algebras which are isomorphic to inverse limits of finite algebras from $\H$ \cite[Corollary 2.4]{CDJP08}. Moreover,
profinite algebras are exactly those algebras which are isomorphic to inverse limits of its finite homomorphic images.
Thus, since varieties are closed under the formation of (finite) homomorphic images, we have the following fact.

\begin{fact}
\label{fact:: standardness for varieties}
Let $\V$ be a variety. Then $\V_{Bc}$ consists of all profinite algebras which have the algebraic reducts in $\V$. Consequently, $\V$ is standard if and only if every Boolean topological algebra with the algebraic reduct in $\V$ is profinite. 
\end{fact}

We say that a topological prevariety $\mathcal G$ is \emph{first-order  axiomatizable} if there exists a set $S$ of first-order sentences such that 
$\mathcal G$ consists of all Boolean topological algebras of type $\sigma$ with algebraic reducts satisfy all sentences from $S$. We are interested in the existence of first-order axiomatization of Boolean cores. Clearly, if $\H$ is a standard universal Horn class, then $\H_{Bc}$ is first-order axiomatizable by any set of first-order sentences defining $\H$. However, there are non-standard universal classes with the first-order axiomatizable Boolean core. Such an universal Horn class generated by a finite lattice is presented in \cite[Example 4.3]{CDJP08}).
We will need the following fact. Recall that an algebra is \emph{locally finite} if every its finitely generated subalgebra is finite. In this paper by a \emph{one-point compactification} we mean a  topological space $(A,{\mathcal T})$ with a distinguish element $0$, where  
\[
{\mathcal T}=\{B\in {\mathcal P}(A)\mid 0\in B\text{ and } A-B \text{ is finite}\}\cup {\mathcal P}(A-\{0\}).
\]
Here $\mathcal P(X)$ denotes the powerset of $X$. The point $0$ is called then the \emph{condensation point} of $(A,{\mathcal T})$. Note that $(A,{\mathcal T})$ is simply the Alexandrov compactification of the discrete spaces on the set $A-\{0\}$. 

\begin{proposition}[\protect{This is a special case of \cite[Second Ultraproduct Technique 5.3]{CDFJ04}}] 
\label{prop::one-point compactification}
Let $\undertilde\A=(\A,{\mathcal T})$ be a Boolean topological algebra such that
\begin{enumerate}
\item $\A$ is locally finite,
\item $\A$ has a constant $0$,
\item $(A,{\mathcal T})$ is a one-point compactification with the condensation point $0$,
\item $\undertilde\A\in \V_{Bt}-\V_{Bc}$,
\item for every algebra $\B$ of the same type as $\A$, if $\B$ is a model of the universal theory of $\A$ and  $(B,{\mathcal T}')$ is the one-point compactification with the condensation point $0$, then the topology ${\mathcal T}'$ is compatible on $\B$ {\rm (}i.e, $(\B,{\mathcal T}')$ is a Boolean topological algebra{\rm )}. 
\end{enumerate}
Then $\V_{Bc}$ is not first-order axiomatizable.
\end{proposition}

\subsection{Determining  principal congruences}

Let us fix a denumerable set $X=\{x,p_0,p_1,p_2,\ldots\}$ of variables. We consider $x$ as a \emph{distinguished}  variable. Let us also fix one  set 
 of terms over $X$ satisfying the following conditions
\begin{itemize}
\item $x\in T_x$,
\item if $\omega$ is a basic operation of positive arity $n$,  then \\ $\omega(p_0,\ldots,p_{i-1},x,p_{i+1},\ldots, p_{n-1})\in T_x$,
\item if $s(x,\bar p), t(x,\bar q)\in T_x$, then $s(t(x,\bar q),\bar p)\in T_x$.
\end{itemize}

From the perspective of this paper it is not relevant which set $T_x$ satisfying the listed conditions is chosen (actually, we could simply take $T_x$ to be the set of all terms over $X$)
. What is important is the following Mal'cev's lemma. For an algebra $\A$ and a pair $a,b$ of its elements let $\theta(a,b)$ denote the principal congruence of $\A$ generated by $(a,b)$.

By a \emph{congruence formula} $\pi(u,v,x,y)$ we mean a first-order formula of the form
\begin{equation}\tag{CF}
\exists \bar p \left( u\approx t_0(z'_0,\bar p)\land\left(\bigwedge_{i=0}^{n-1} t_i(z_i',\bar p)\approx t_{i+1}(z'_{i+1},\bar p)\right) \land t_n(z'_n,\bar p)\approx v  \right),
\end{equation}
where $t_i(x,\bar p)\in T_x$ and $\{z_i,z_i'\}=\{x,y\}$ for all $0\leq i\leq n$.  Let us denote the set  $\{t_0,\ldots,t_n\}$ of terms appearing in (CF) by $\term(\pi)$. The set of all congruence formulae will be denoted by $\Pi$. (Now we formalize the notions from the introduction: The length of $\pi$ in (CF) equals $n$, and the depth is the maximal depth of $x$ in the terms $t_0,\ldots,t_n$.)

\begin{lemma}[\protect{\cite[Theorem V.3.3]{BS81}}] \label{lem:: Mal'cev}
For an algebra $\A$ and elements $a,b,c,d\in A$ we have $(c,d)\in\theta(a,b)$ if and only if there is a  congruence formula $\pi(u,v,x,y)$ such that $\A\models \pi(c,d,a,b)$.
\end{lemma} 

The above lemma suggests that we may impose various restrictions on determining principal congruences. For instance, a class $\mathcal C$ of algebras has \emph{definable principal congruences} (DPC in short) if there is a finite set $P\subseteq \Pi$ such that $\theta(a,b)=\{(c,d)\in A^2\mid \exists \pi\in P\;\;\A\models  \pi(c,d,a,b)\}$ for all $a,b\in A$, $\A\in\mathcal C$. (This topic is covered in a more general context of relative congruences in \cite{CD96}.) In this paper a weaker restriction will be crucial:
For $F\subseteq T_x$ let 
\[
\Pi^F=\{\pi\in \Pi\mid \term(\pi)\subseteq F\}.
\] 
For an algebra $\A$ and its elements $a,b$ define
\[
\theta^F(a,b)=\{(c,d)\in A^2\mid\exists \pi\in\Pi^F\;\A\models(c,d,a,b)\}.
\]
We say that a subset $F$ of $T_x$ \emph{determines principal congruences} in a class $\mathcal C$ if for every $\A\in\mathcal C$ and $a,b\in A$ we have $\theta(a,b)=\theta^F(a,b)$.
We say that $\mathcal C$ has \emph{term finite principal congruences} (TFPC in short) if there is a finite subset of $T_x$ which determines principal congruences in $\mathcal C$.

For $F\subseteq T_x$ and an algebra $\A$ let $F^{\A}$ be the set of all translation on $\A$ induced by terms in $F$. Formally,
 \[
F^\A=\{f\in A^A\mid \exists t(x,\bar p)\in F,\bar e\in A^n\;\;\forall a\in A \;f(a)=t(a,\bar e)\}.
\]
For an equivalence relation $\theta$ on the carrier of an algebra $\A$ let 
\[
\theta_F=\{(a,b)\in A^2\mid \forall f\in F^\A\; (f(a),f(b))\in \theta\}.
\]
Note that $\theta_F$ is an equivalence relation on $A$. If $x\in F$, then it is contained in $\theta$. More generally, $\theta_G\subseteq \theta_F$ whenever $F\subseteq G$.  Moreover, $\theta_{T_x}$ is the largest congruence on $\A$ contained in $\theta$  \cite[Lemma 2.1]{CDFJ04}). It is called the \emph{syntactic congruence of} $\theta$ and is denoted by $\syn(\theta)$.

We say that $F\subseteq T_x$ \emph{determines syntactic congruences} in a class $\mathcal C$ of algebras provided that $\syn(\theta)=\theta_F$ for every equivalence relation $\theta$ on the carrier of an algebra in $\mathcal C$. We say that $\mathcal C$ has \emph{finitely determined syntactic congruences}  if there is a finite subset of $T_x$ determining syntactic congruences in $\mathcal C$. 

It appears that the properties of having $FDSC$ and $TFPC$ are equivalent.
\begin{lemma}[\protect{\cite[Lemma 2.3]{CDFJ04}}]\label{FDSC <=> TFPC}
Let $\mathcal C$ be a class of algebras and $F\subseteq T_x$. Then $F$ determines syntactic congruences in $\mathcal C$ if and only if $F$ determines principal congruences in $\mathcal C$.
\end{lemma}

From the proof of Lemma \ref{FDSC <=> TFPC} we may extract the following fact.

\begin{lemma}\label{lem:: theta^F subseteq theta}
Let $\A$ be an algebra, $a,b\in A$, and $\theta$ be an eqivalence relation on $A$. If $(a,b)\in\theta_F$, then $\theta^F(a,b)\subseteq \theta$.
\end{lemma}

What makes the property of having FDSC relevant in the context of topological algebras is the fact that this property yields  profiniteness. The following observation was first proved for many special cases, see the historical comments preceding \cite[Clopen Equivalence Lemma 4.2]{CDFJ04}. (The reader may also see \cite{SZ17} for a related characterization of profiniteness.)

\begin{proposition}[\protect{\cite[Theorem 2.13]{CDJP08}}]
\label{prop::FDSC=>profiniteness}
Let $\undertilde\A$ be a Boolean topological algebra. If there is a finite set of terms that determines syntactic congruences on $\A$, then $\undertilde\A$ is profinite.
\end{proposition}

Proposition \ref{prop::FDSC=>profiniteness} and Fact \ref{fact:: standardness for varieties} give the following theorem.

\begin{theorem}[\protect{\cite[Theorem 2.13]{CDJP08}, see also \cite[FDSC-HSP Theorem 4.3]{CDFJ04}}]\label{thm:: FDSC => standard}
Let $\V$ be a variety. If $\V$ has FDSC, then $\V$ is standard. 
\end{theorem}

\section{Determining  principal subcongruences}
\label{sec::TFPSC}

We say that a pair $F,G$ of subsets of $T_x$ \emph{determines principal subcongruences} in a class $\mathcal C$ of algebras if for every $\A\in \mathcal C$ and for every pair $a,b$ of distinct elements of $A$ there is a pair $c,d$ of distinct elements of $A$ such that  
\[
(c,d)\in\theta^F(a,b)\quad\text{ and }\quad \theta(c,d)=\theta^G(c,d).
\]
A class $\mathcal C$ has \emph{term finite principal subcongruences} (TFPSC in short) if there is a pair  of finite subsets of $T_x$ which determines principal subcongruenses in $\mathcal C$.

Let $\theta,\eta$ be equivalence relations on a set $A$. If $\eta\subseteq\theta$ then the quotient equivalence relation $\theta/\eta$ on the set $A/\eta$ is given by
\[
(a/\eta,b/\eta)\in \theta/\eta\quad\text{iff}\quad (a,b)\in\theta.
\]

\begin{lemma} \label{lem::syn and F commute}
Let $F\subseteq T_x$, $\theta$ be an equivalence relation on the carrier of an algebra $\A$ and $a,b\in A$. Then $(a,b)\in \theta_F$ if and only if $(a/\kern-1pt\syn(\theta),b/\kern-1pt\syn(\theta))\in \left(\theta/\kern-1pt\syn(\theta)\right)_F$. Consequently, $\syn(\theta/\kern-1pt\syn(\theta))$ is the identity relation on $A/\kern-1pt\syn(\theta)$.
\end{lemma}

\begin{proof}
By definition of $\theta_F$, we have

\begin{align*}
 (a,b)\in \theta_F &\quad\text{iff}\quad (f(a),f(b))\in\theta\;\,\text{for all}\;\, f\in F^\A\\
&\quad\text{iff}\quad (f(a)/\kern-1pt\syn(\theta),f(b)/\kern-1pt\syn(\theta))\in \theta/\kern-1pt\syn(\theta)\;\, \text{for all}\;\, f\in F^\A\;\\
&\quad\text{iff}\quad (g(a/\kern-1pt\syn(\theta)),g(b/\kern-1pt\syn(\theta)))\in \theta/\kern-1pt\syn(\theta)\;\, \text{for all}\;\, g\in F^{\A/\kern-1pt\syn(\theta)}\\
&\quad\text{iff}\quad (a/\kern-1pt\syn(\theta),b/\kern-1pt\syn(\theta))\in \left(\theta/\kern-1pt\syn(\theta)\right)_F.
\end{align*}
The third equivalence follows from the fact that $\syn(\theta)$ is a congruence on $\A$. 
By setting $F=T_x$, we obtain the second assertion.
\end{proof}

For $F,G\subseteq T_x$ we define
\[
F\circ G =\{ s(t(x,\bar q),\bar p)\mid s(x,\bar p)\in F\text{ and } t(x,\bar q)\in G\}.
\]
Then for every algebra $\A$ and every equivalence relation $\theta$ on $A$ we have
\begin{equation} \label{eqn:: F circ G}\tag{COMP}
\theta_{F\circ G}=(\theta_F)_G.
\end{equation}

\begin{proposition}\label{prop:: main}
Let $\mathcal C$ be a class of algebras and  assume that $\mathcal C$ is closed under homomorphic images.
Let  $F,G\subseteq T_x$. If the pair $F,G$ determines principal subcongruences in $\mathcal C$, then $G\circ F$ determines syntactic congruences in $\mathcal C$. 
\end{proposition}

\begin{proof}
Let $\A$ be an algebra in $\mathcal C$ and $\theta$ be an equivalence relation on $A$. We want to show that $\syn(\theta)=\theta_{G\circ F}$. Let $\B=\A/\kern-1pt\syn(\theta)$ and $\eta=\theta/\kern-1pt\syn(\theta)$. By Lemma \ref{lem::syn and F commute}, $\syn(\eta)$ is the identity relation on $B$. Thus we should show that $\eta_{G\circ F}$ is the identity relation on $B$.
Assume that it is not the case, i.e., there is a pair $a,b$ of distinct elements of $B$ such that $(a,b)\in\eta_{G\circ F}$. 

By assumption, $\B=\A/\kern-1pt\syn(\theta)\in\mathcal C$.
Thus the pair $F,G$ determines principal subcongruences in $\B$. This yields that there exists a pair $c,d$ of distinct elements such that
\[
(c,d)\in\theta^F(a,b)\quad\text{ and }\quad\theta(c,d)=\theta^G(c,d).
\]
By (\ref{eqn:: F circ G}), $\eta_{G\circ F}=(\eta_G)_F$. Hence, by Lemma \ref{lem:: theta^F subseteq theta}, 
\[
(c,d)\in\theta^F(a,b)\subseteq \eta_G.
\]
Applying Lemma \ref{lem:: theta^F subseteq theta} once more, but 
for $c,d$ and $G$, we infer that
\[
\theta(c,d)=\theta^G(c,d)\subseteq \eta.
\]
It follows that $(c,d)\in\eta_{T_x}$. Since $\eta_{T_x}=\syn(\eta)$ and $\syn(\eta)$ is the identity relation on $B$, $c=d$. This is a contradiction with the choice of $c,d$. 
\end{proof}

\begin{theorem}\label{thm:: FDPC <=> TFPSC}
Let $\mathcal C$ be a class of algebras and  assume that $\mathcal C$ is closed under homomorphic images.
Then $\mathcal C$ has FDSC if and only if $\mathcal C$ has TFPSC.
\end{theorem}

\begin{proof}
The forward implication follows from Lemma \ref{FDSC <=> TFPC} and the obvious fact that having TFPC yields TFPSC. The backward implication follows from Proposition \ref{prop:: main}.
\end{proof}

We present two corollaries.

We say that a class $\mathcal C$ has \emph{definable principal subcongruences} (DPSC for short) if there are finite subsets  $P,R$ of $T_x$ such that for every algebra $\A\in\mathcal C$ and for every pair $a,b$ of distinct elements of $A$ there are a pair $c,d$ of distinct elements of $A$ and a congruence formula $\pi\in P$ such that
\[
\A\models\pi(c,d,a,b)\quad\text{ and }\quad\theta(c,d)=\{(e,f)\in A^2\mid\exists \rho\in R\;\, \A\models\rho(e,f,c,d)\}.
\] 
In \cite[Proposition 2.8]{CDFJ04} it is pointed that having DPC yiels having TFPC, and hence also having FDSC. But 
clearly, also having DPSC yields TFPSC. Thus Theorem \ref{thm:: FDPC <=> TFPSC} gives the following stronger fact.

\begin{corollary} \label{cor::DPSC=>FDSC}
Let $\mathcal C$ be a class of algebras and  assume that $\mathcal C$ is closed under homomorphic images. If $\mathcal C$ has DPSC, then $\mathcal C$ has FDPC.
\end{corollary}

The following corollary was proved in \cite{Wan90}, and then reproved in \cite{DJMM08}. The novelty here is the fact that we do not need to use  J\'onsson's terms in order to prove it.
Indeed, one of the main results in \cite{BW02} is the proof (without J\'onsson's terms) that finitely generated congruence distributive varieties have DPSC. 
 
\begin{corollary}\label{fg+cd=>FDSC}
Let $\mathcal V$ be a congruence distributive finitely generated variety. Then $\V$ has FDSC.
\end{corollary}

\begin{proof}
It follows by Corollary \ref{cor::DPSC=>FDSC} and \cite[Theorem 2]{BW02}.
\end{proof}

\section{Undecidability}
\label{sec:: undecidability}

By modifying McKenzie construction \cite{McK96}, Moore \cite{Moo15} provided an effective construction that takes a Turing machine $\T$ and returns the algebra $\sf A'(\T)$ with some properties depending on whether $\T$ halts or not. The construction  is quite complicated. However, we extract some information relevant for our paper in the next theorem. 

We say that two algebras are term equivalent if they have the same carrier and the same set of term operations. This means that every basic operation of one of this algebras is a term operation of the second algebra. 

Let ${\bf Q}_\omega=(Q_\omega,\sqcap,\cdot,0)$ be the algebra with the carrier
\[
Q_\omega=\{a_0,b_0,a_1,b_1,\ldots\}\cup\{0\},
\]
where $0$ is a constant and $\sqcap,\cdot$ are binary operations given by 
\[
a\sqcap b =
\begin{cases}
a &\text{ if } a=b\\
0 &\text{ if }a\not =b
\end{cases}
\quad\quad\quad
a\cdot b =
\begin{cases}
b_i &\text{ if } a=a_i,b\,=b_{i+1}\text{ and }i\in \mathbb N\\ 
0 &\text{ otherwise}
\end{cases}.
\]
Note that $\bf Q_\omega$ is subdirectly irreducible and locally finite.

\begin{theorem}[\protect{\cite{Moo15}}]
\label{thm::Moore}
Let $\T$ be a Turing machine and $\sf A'(T)$ be the algebra constructed as in \cite{Moo15} for $\T$. Then
\begin{itemize}
\item if $\T$ halts on empty input, then ${\sf V}(\sf A'(\T))$ has DPSC;
\item if $\T$ does not halt on empty input, then there is an algebra $\bf Q$ in  ${\sf V}(\sf A'(\T))$ which is term equivalent to $\bf Q_\omega$.
\end{itemize} 
\end{theorem}

The following fact holds.

\begin{proposition}
\label{prop::when T does not halt}
Assume that there is an algebra $\bf Q$ in $\V$ which is term equivalent to $\bf Q_\omega$. Then $\V_{Bc}$ is not  first-order axiomatizable.
\end{proposition}

\begin{proof} 
We equip the set $Q_\omega$ with the Boolean topology $\mathcal T$ such that $(Q_\omega,\mathcal T)$ is the a one-point compactification with the condensation point $0$. It is the case that $\mathcal T$ is compatible with basic operations of $\bf Q_\omega$, and hence with basic operations of $\bf Q$. This means that $\undertilde{\bf Q}=(\bf Q,\mathcal T)$ is a Boolean topological algebra.  
We use Proposition \ref{prop::one-point compactification} for $\undertilde{\bf Q}$. The condition (1)-(3) from this theorem thus hold. The argument for the condition (4), which says that $\undertilde{\bf Q}$ witnesses non-standardness for $\V$, was given in \cite[Example 7.7]{CDFJ04}: Since $\bf Q$ is infinite subdirectly irreducible, $\undertilde{\bf Q}$ cannot be profinite. 

Let us check the condition (5). Let $\bf Q'$ be an algebra of the same type as $\bf Q$ and assume that $\bf Q'$ satisfies every universal sentence which holds in $\bf Q$. Let ${\mathcal T}'$ be the Boolean topology such that $(Q',{\mathcal T}')$ is the one-point compactification with the condensation point $0$.
In $\bf Q'$ we have (term) operations $\cdot$, $\sqcap$ and $0$. We have to verify their continuity. Indeed, all basic operations in $\bf Q'$ may be obtained by composition of these three ones. Thus their continuity will follow.

The operation $0$, as a constant, is obviously continuous. Let $C\in {\mathcal T}'$ and $a,b\in Q'$ be such that $a\circ b \in C$, where $\circ=\cdot$ or $\circ =\sqcap$. We have to find two sets $A,B$ in ${\mathcal T}'$ such that $(a,b)\in A\times B$ and $A\circ B\subseteq C$.

In order to do it let us first observe that $\bf Q$, and hence also $\bf Q'$, satisfy the sentence
\[
(\forall x,y)\;  x\circ y \not\approx 0 \,\to\, ( x\not\approx 0\,\land\, y\not\approx 0). 
\] 
Thus, in the case when $0\not\in C$ we may put $A=\{a\}$ and $B=\{b\}$. 

Let us assume that $0\in C$. Then the set $Q'-C$ is finite and does not contain $0$. Note that $\bf Q$ and $\bf Q'$ satisfy the sentence
\[
(\forall x,y,x',y')\; (x\circ y\approx x'\circ y'\,\land\, x\circ y \not\approx 0) \,\to\, (x\approx x'\,\land\, y\approx y'\,\land\, x\not\approx 0\,\land\, y\not\approx 0). 
\] 
Hence the cardinalities of the sets
\begin{align*}
L =&\; \{a\in Q'\mid \exists d\in Q' \; a\circ d\in Q'-C\},\\
R =& \;\{b\in Q'\mid \exists c\in Q' \; c\circ b\in Q'-C\}
\end{align*}
are not greater than the cardinality of $Q'-C$. Thus they are finite. Moreover,  $0,a\not\in L$ and $0,b\not\in R$. Hence we may put $A=Q'-L$ and $B=Q'-R$.
\end{proof}

\begin{theorem}
\label{thm::A'(T)}
Let $\T$ be a Turing machine and $\A'(T)$ be the algebra constructed as in \cite{Moo15} for $\T$. Then
\begin{itemize}
\item if $\T$ halts on empty input, then ${\sf V}(\sf A' (\T))$ has FDSC;
\item if $\T$ does not halt on empty input, then  ${\sf V}(\sf A'(\T))_{Bc}$ is not  first-order axiomatizable.
\end{itemize} 
\end{theorem}

\begin{proof}
It follows from Corollary \ref{cor::DPSC=>FDSC}, Theorem \ref{thm::Moore} and Proposition \ref{prop::when T does not halt}.
\end{proof}

\begin{proof}[Proof of Theorems \ref{thm::main standard}, \ref{thm::main FDSC} and \ref{thm::main FO}]
By Theorem \ref{thm:: FDSC => standard}, having FDSC yields standardness for varieties and, clearly, standardness yields first-order axiomatizability of the Boolean core. Thus the theorem follows from the undeciability of halting problem and Theorem \ref{thm::A'(T)}.
\end{proof} 

\bibliographystyle{plain}
\bibliography{FDSC20170715}
\end{document}